\documentclass[a4paper,11pt,twoside]{article}																																																																																																												
\usepackage{amsmath,amssymb,amsfonts,amsthm}
\usepackage{layout,textcomp}
\newtheorem{Theorem}{Theorem}
\newtheorem{Prop}{Proposition}
\newtheorem{Cor}{Corollary}
\newtheorem*{Min}{Min-Oo' s Conjecture}
\newtheorem*{Claim}{Claim}

\theoremstyle{remark}

\newtheorem*{Ak}{Acknowledgements}
\DeclareMathOperator{\Ric}{Ric}
\DeclareMathOperator{\Vol}{Vol}
\title{RIGIDITY OF AREA-MINIMIZING HYPERBOLIC SURFACES IN THREE-MANIFOLDS}
\author{IVALDO NUNES}
\date{}

\begin{document}
\maketitle
\begin{abstract}
We prove that if $M$ is a three-manifold with scalar curvature greater than or equal to -2 and $\Sigma\subset M$ is a two-sided compact embedded Riemann surface of genus greater than 1 which is locally area-minimizing, then the area of $\Sigma$ is greater than or equal to $4\pi(g(\Sigma)-1)$, where $g(\Sigma)$ denotes the genus of $\Sigma$. In the equality case, we prove that the induced metric on $\Sigma$ has constant Gauss curvature equal to -1 and locally $M$ splits along $\Sigma$. As a corollary, we obtain a rigidity result for cylinders $(I\times\Sigma,dt^2+g_{\Sigma})$, where $I=[a,b]\subset\mathbb{R}$ and $g_{\Sigma}$ is a Riemannian metric on $\Sigma$ with constant Gauss curvature equal to -1.
\end{abstract}

\section{Introduction}

It is an interesting fact in differential geometry that if $M$ is a three-manifold with lower bounded scalar curvature then the existence of an 
area-minimizing surface can influence the geometry of $M$.

For instance, it was shown by R. Schoen and S. T. Yau \cite{SY} that if $M$ is a compact orientable three-manifold with nonnegative scalar curvature and $\Sigma\subset M$ is an incompressible two-torus (i.e., the fundamental group of $\Sigma$ injects into that of $M$), then $M$ is flat. To prove that result, they first show that any such manifold contains a stable minimal two-torus. Next, they observe, using the second variation formula of area, that if $M$ has positive scalar curvature, then every compact stable minimal surface in $M$ is a two-sphere. The result follows because if $M$ admits a non-flat metric of nonnegative scalar curvature, then $M$ also admits a metric of positive scalar curvature (see \cite{KW}). 

In \cite{FCS}, D. Fischer-Colbrie and R. Schoen conjectured that in the Schoen and Yau's result above it is sufficient that $M$ contains an area-minimizing two-torus (not necessarily incompressible). This conjecture was proved in \cite{CG}, by M. Cai and G. Galloway. They proved that if $M$ has nonnegative scalar curvature and $\Sigma\subset M$ is a two-sided embedded two-torus which is area-minimizing in its isotopy class, then $M$ is flat. This result is obtained as a corollary of the following local statement.

\begin{Theorem}[M. Cai, G. Galloway]\label{Thm0}
Let $(M,g)$ be a three-manifold with nonnegative scalar curvature. If $\Sigma$ is a two-sided embedded two-torus in $M$ which is locally area-minimizing, then $M$ is flat in a neighborhood of $\Sigma$. 
\end{Theorem}

It follows that the induced metric on $\Sigma$ is flat and that locally $M$ splits along $\Sigma$. The proof of Theorem \ref{Thm0} uses an argument based on a local deformation around $\Sigma$ to obtain a metric with positive scalar curvature.

Recently, H. Bray, S. Brendle and A. Neves studied in \cite{BBN} the case where $M$ has scalar curvature greater than or equal to $2$ and $\Sigma\subset M$ is a locally area-minimizing embedded two-sphere. In their case, the model is the Riemannian manifold $(\mathbb{R}\times S^2,dt^2+g)$, where $g$ is the standard metric on $S^2$ with constant Gauss curvature equal to 1. They proved the following result.

\begin{Theorem}[H.Bray, S. Brendle, A. Neves]\label{Thm0a}
Let $(M,g)$ be a three-manifold with scalar curvature $R_g\geqslant 2$. If $\Sigma$ is an embedded two-sphere which is locally area-minimizing, then $\Sigma$ has area less than or equal to $4\pi$. Moreover, if equality holds, then $\Sigma$ with the induced metric has constant Gauss curvature equal to 1 and locally $M$ splits along $\Sigma$.  
\end{Theorem}

The proof in \cite{BBN} is based on a construction of a one-parameter family of constant mean curvature two-spheres. A global result was also obtained using the local one above. More precisely, it was proved that if $\Sigma$ is area-minimizing in its homotopy class and has area equal to $4\pi$, then the universal cover of $M$ is isometric to $(\mathbb{R}\times S^2,dt^2+g)$. A similar rigidity result for area-minimizing projective planes was obtained in \cite{BBEN}.


A natural question is to know what happens when the model is the Riemannian product manifold $(\mathbb{R}\times \Sigma, dt^2+g_{\Sigma})$, where $\Sigma$ is a Riemann surface of genus greater than 1 and $g_{\Sigma}$ is a Riemannian metric on $\Sigma$ with constant Gauss curvature equal to -1. 

In the present paper, we prove that the analogous result is true in this case. The first theorem of this paper is stated below.

\begin{Theorem}\label{Thm1}
Let $(M^3,g)$ be a Riemannian manifold with $R_g\geqslant -2$, where $R_g$ denotes the scalar curvature of $M$. If $\Sigma^2\subset M$ is a two-sided compact embedded Riemann surface of genus $g(\Sigma)\geqslant 2$ which is locally area-minimizing, then
\begin{equation}
|\Sigma|_g\geqslant 4\pi(g(\Sigma)-1),\label{inequality}
\end{equation}
where $|\Sigma|_g$ is the area of $\Sigma$ with respect to the induced metric. Moreover, if the equality holds, then $\Sigma$ has a neighborhood which is isometric to \linebreak $((-\epsilon,\epsilon)\times\Sigma,dt^2+g_{\Sigma})$, where $\epsilon>0$ and $g_{\Sigma}$ is the induced metric on $\Sigma$ which has constant Gauss curvature equal to -1. More precisely, the isometry is given by $f(t,x)=\exp_x(t\nu(x)), (t,x)\in(-\epsilon,\epsilon)\times\Sigma$, where $\nu$ is the unit normal vector field to $\Sigma$.
\end{Theorem}

We note that a related rigidity result for constant mean curvature surfaces of genus 1 was obtained in \cite{ACG}. We also refer the reader to the excellent surveys \cite{G} and \cite{SB} on rigidity problems associated to scalar curvature. 

Let us give an idea of the proof of Theorem \ref{Thm1}. The inequality (\ref{inequality}) follows from the second variation of area using the Gauss equation, the lower bound of the scalar curvature and the Gauss-Bonnet theorem. In the equality case, we construct, using the implicit function theorem, a one-parameter family of constant mean curvature surfaces, denoted by $\Sigma_t$, with $\Sigma_0=\Sigma$ and all having the same genus. The next argument in the proof is the fundamental one. Arguing by contradiction and using the solution of the Yamabe problem for compact manifolds with boundary and the Hopf's maximum principle, we are able to conclude that each $\Sigma_t$ has the same area. Finally, we obtain from this that $\Sigma$ has a neighborhood isometric to $((-\epsilon,\epsilon)\times\Sigma, dt^2+g_{\Sigma}).$ 

If we suppose that $\Sigma$ minimizes area in its homotopy class, then we obtain global rigidity using a standard continuation argument contained in \cite{BBN,CG}. 

\begin{Cor}\label{Cor1}
Let $(M^3,g)$ be a complete Riemannian three-manifold with \linebreak$R_g\geqslant -2$. Moreover, suposse that $\Sigma\subset M$ is a two-sided compact embedded Riemann surface of genus $g(\Sigma)\geqslant 2$ which minimizes area in its homotopy class. Then $|\Sigma|_g$ satisfies the inequality (\ref{inequality}) and if equality holds, then \linebreak $(\mathbb{R}\times\Sigma,dt^2+g_{\Sigma})$ is an isometric covering of $(M^3,g)$, where $g_{\Sigma}$ is the induced metric on $\Sigma$ which has constant Gauss curvature equal to -1. The covering is given by $f(t,x)=\exp_x(t\nu(x)), (t,x)\in\mathbb{R}\times\Sigma$, where $\nu$ is the unit normal vector to $\Sigma$.
\end{Cor}

In recent years, several results concerning the problem of recognizing the geometry of a compact manifold with boundary provided the geometry of the boundary is known and some curvature conditions are satisfied were obtained. For example, in \cite{Mi}, P. Miao observed that the positive mass theorem (see \cite{SY2,EW}) implies the following rigidity result for the unit ball $B^n\subset \mathbb{R}^n$.

\begin{Theorem}[P. Miao]
Let $g$ be a smooth Riemannian metric on $B^n$ with nonnegative scalar curvature such that $\partial B^n=S^{n-1}$ with the induced metric has mean curvature greater than or equal to $n-1$ and is isometric to $S^{n-1}$ with the standard metric. Then $g$ is isometric to the standard metric of $B^n$. 
\end{Theorem}

The theorem above was generalized by Y. Shi and L. Tam in \cite{ST}. There are some analogous rigidity results for the hyperbolic space (see \cite{M}, \cite{AD}, \cite{CH} and \cite{W}). A similar rigidity for the hemisphere $S^n_+$ was conjectured by M. Min-Oo in \cite{M2}:
\begin{Min}
Let $g$ be a smooth metric on the hemisphere $S^{n}_+$ with scalar curvature $R_g\geqslant n(n-1)$ such that the induced metric on $\partial S^n_+$ agrees with the standard metric on $\partial S^n_+$ and is totally geodesic. Then $g$ is isometric to the standard metric on $S^n_+$.
\end{Min}

This conjecture is true for $n=2$, in which case it follows by a theorem of Toponogov \cite{Top} (see also \cite{HW}). Recently, counterexamples were constructed by S. Brendle, F. C. Marques and A. Neves in \cite{BMN} for $n\geqslant 3$.  
We refer the reader to \cite{HW,ME,HuWu} for partial results concerning the Min-Oo's conjecture. In \cite{BM}, a rigidity result for small geodesic balls of $S^n$ was proved.

The next theorem is a rigidity result for cylinders $([a,b]\times\Sigma,dt^2+g_{\Sigma})$, where $\Sigma$ is a Riemann surface of genus greater than 1 and constant Gauss curvature equal to -1. This is the analogue of Miao's result and Min-Oo's Conjecture in our setting.

Recall that a three-manifold is irreducible if every embedded 2-sphere in $M$ bounds a 3-ball embedded in $M$.

\begin{Theorem}\label{Cor2}
Let $\Sigma$ be a compact Riemann surface of genus $g(\Sigma)\geqslant 2$ and $g_{\Sigma}$ a metric on $\Sigma$ with $K_{\Sigma}\equiv -1$. Let $(\Omega^3,g)$ be a compact orientable irreducible connected Riemannian three-manifold with boundary satisfying the following properties:
\begin{itemize}
\item $R_g\geqslant -2$.
\item $H_g\geqslant 0$.($H_g$ is the mean curvature of $\partial\Omega$ and the convention for the mean curvature is $\vec{H}_g=-H_g\eta$, where $\vec{H}$ is the mean curvature vector and $\eta$ is the outward normal vector).
\item Some connected component of $\partial \Omega$ is incompressible in $\Omega$ and with the induced metric is isometric to $ (\Sigma,g_{\Sigma}).$
\end{itemize}
Moreover, suppose that $\Omega$ does not contain any one-sided compact embedded surface. Then $(\Omega,g)$ is isometric to $([a,b]\times\Sigma,dt^2+g_{\Sigma}).$
\end{Theorem}

We note that the similar result for cylinders $[a,b]\times S^2$, where $S^2$ is the round unit sphere, does not hold. In fact, consider a rotationally symmetric metric $g=u(t)^4(dt^2+g_{S^2})$ on $\mathbb{R}\times S^2$ with constant scalar curvature equal to $2$ such that $u(0)>1$ and $u^{\prime}(0)=0$ (see \cite{RS}). Choosing $a>0$ such that $u(a)=u(0)$, we have that the rescaled metric $\overline{g}=u(0)^{-4}\,g$ on $[0,a]\times S^2$ gives a counterexample. 

The following example justifies the requirement that $\Omega$ does not contain any one-sided compact embedded surface. Let $(\hat{\Sigma},g_{\hat{\Sigma}})$ be a compact non-orientable surface with constant Gauss curvature equal to -1. Denote by $\Sigma$ the orientable double covering of $\hat{\Sigma}$ and by $\pi$ the covering map. Now, let $g_{\Sigma}=\pi^*g_{\hat{\Sigma}}$ and consider $(M=[-k,k]\times\Sigma,g=dt^2+g_{\Sigma})$. Take the subgroup $\Gamma=\{id,f\}\subset\textrm{Iso}(M,g)$, where $f$ is defined by $f(t,x)=(-t,\phi(x))$ and $\phi\in\textrm{Iso}(\Sigma,g_{\Sigma})$ is the non-trivial deck transformation of $\pi:\Sigma\longrightarrow\hat{\Sigma}$. Now, consider the Riemannian manifold $(\Omega,g_{\Omega})$, where $\Omega=M/\Gamma$ and $g_{\Omega}$ is the quotient metric. Note that $\Omega$ is orientable and irreducible, $R_{g_{\Omega}}=-2$, $H_{g_{\Omega}}=0$, $\partial\Omega$ is incompressible in $\Omega$ and with the induced metric is isometric to $(\Sigma,g_{\Sigma})$. Finally, observe that $\partial\Omega$ has only one component and that the image of $\{0\}\times\Sigma$ is a one-sided compact embedded surface in $\Omega$.  
{\small
\begin{Ak} This work is part of the author's Ph.D. thesis at IMPA. The author would like to thank his advisor Fernando Cod\' a Marques for his constant encouragement and for all stimulating discussions during the preparation of this work. The author was financially suported by CNPq-Brazil and FAPERJ-Brazil.
\end{Ak}}

\section{Proof of the inequality (\ref{inequality})}

Let $\nu$ be the unit normal vector field to $\Sigma$. For each function $\phi\in C^{\infty}(\Sigma)$, we have, by the second variation formula of area and the fact that $\Sigma$ is locally area minimizing, that
$$
\int_{\Sigma}(\Ric(\nu,\nu)+|A|^2)\,\phi^2\,d\sigma\leqslant\int_{\Sigma}|\nabla\phi|^2\,d\sigma,
$$
where $A$ and $d\sigma$ denote the second fundamental form and the area element of $\Sigma$, respectively. Choosing $\phi=1$, we obtain
\begin{equation}
\int_{\Sigma}(\Ric(\nu,\nu)+|A|^2)\,d\sigma\leqslant 0. \label{inequality2}
\end{equation}

Now, the Gauss equation implies
\begin{equation}
\Ric(\nu,\nu)=\dfrac{1}{2}R_g-K_{\Sigma}-\dfrac{1}{2}|A|^2, \label{inequality3}
\end{equation}
where $K_{\Sigma}$ denotes the Gauss curvature of $\Sigma$.

Substituting (\ref{inequality3}) in (\ref{inequality2}), we get
\begin{equation}
\dfrac{1}{2}\int_{\Sigma}(R_g+|A|^2)\,d\sigma\leqslant\int_{\Sigma}K_{\Sigma}\,d\sigma. \label{inequality4}
\end{equation}

By the Gauss-Bonnet theorem and the fact that $R_g\geqslant -2$ and $|A|^2\geqslant 0$, we have
$$
-|\Sigma|_g\leqslant 4\pi(1-g(\Sigma)).
$$

Therefore,  $|\Sigma|_g\geqslant 4\pi(g(\Sigma)-1).$

\section{Equality Case}
\begin{Prop}\label{Prop1}
If $\Sigma$ attains the equality in (\ref{inequality}), then $\Sigma$ is totally geodesic. Moreover, $\Ric(\nu,\nu)= 0$ and $R_g= -2$ on $\Sigma$ and $\Sigma$ has constant Gauss curvature equal to -1 with the induced metric.
\end{Prop}
\begin{proof}
If $|\Sigma|_g=4\pi(g(\Sigma)-1)$, then it follows from the proof of (\ref{inequality}) that the inequalities (\ref{inequality2}) and (\ref{inequality4}) are in fact equalities. The equality in (\ref{inequality2}) together with the stability of $\Sigma$ implies that the constant functions are in the kernel of the Jacobi operator $L=\Delta_{\Sigma}+Ric(\nu,\nu)+|A|^2$ of $\Sigma$. Therefore, $Ric(\nu,\nu)+|A|^2=0$ on $\Sigma$.

Now, the equality in (\ref{inequality4}) implies that $R_g=-2$ and $A=0$ on $\Sigma$. Finally, by (\ref{inequality3}), we conclude that $\Sigma$ has constant Gauss curvature equal to -1 with the induced metric.
\end{proof}

The construction in the next proposition is fundamental to conclude the rigidity in the Theorem \ref{Thm1}. The same construction was used in \cite{ACG} and \cite{BBN} to prove similar rigidity results. We prove it here for completeness.  

\begin{Prop}\label{prop2}
If $\Sigma$ attains the equality in (\ref{inequality}), then there exists $\epsilon>0$ and a smooth family $\Sigma_t\subset M$, $t\in(-\epsilon,\epsilon)$, of compact embedded surfaces satisfying:
\begin{itemize}
\item $\Sigma_t=\{\textrm{exp}_x(w(t,x)\nu(x)):x\in M\}$, where $w:(-\epsilon,\epsilon)\times \Sigma\longrightarrow \mathbb{R}$ is a smooth function such that $$w(0,x)=0, \dfrac{\partial w}{\partial t}(0,x)=1\, \, \textrm{and}\,\, \int_{\Sigma}(w(t,\cdot)-t)\,d\sigma=0.$$ 
\item $\Sigma_t$ has constant mean curvature for all $t\in(-\epsilon,\epsilon).$
\end{itemize}
\end{Prop}
\begin{proof}
By the previous proposition, we have $L=\Delta_{\Sigma}.$ Fix $\alpha\in(0,1)$ and consider the Banach spaces $X=\{u\in C^{2,\alpha}(\Sigma):\int_{\Sigma}u\,d\sigma=0\}$ and $Y=\{u\in C^{0,\alpha}(\Sigma):\int_{\Sigma}u\,d\sigma=0\}$. For each real function $u$ defined on $\Sigma$, let $\Sigma_u=\{\textrm{exp}_x(u(x)\nu(x)):x\in\Sigma\}$, where $\nu$ is the unit normal vector field to $\Sigma$.

Choose $\epsilon>0$ and $\delta>0$ such that $\Sigma_{u+t}$ is a compact surface of class $C^{2,\alpha}$ for all $(t,u)\in (-\epsilon,\epsilon)\times B(0,\delta)$, where $B(0,\delta)=\{u\in C^{2,\alpha}(\Sigma):\Arrowvert u\Arrowvert_{C^{2,\alpha}(\Sigma)}<\delta\}$. Denote by $H_{\Sigma_{u+t}}$ the mean curvature of $\Sigma_{u+t}$.

Now, consider the application $\Psi:(-\epsilon,\epsilon)\times B(0,\delta)\subset X\longrightarrow Y$ defined by
$$
\Psi(t,u)=H_{\Sigma_{u+t}}-\dfrac{1}{|\Sigma|}\int_{\Sigma}H_{\Sigma_{u+t}}\,d\sigma.
$$
Notice that $\Psi(0,0)=0$ because $\Sigma_{0}=\Sigma.$

The next step is to compute $D\Psi(0,0)\cdot v$, for $v\in X$. We have 
\begin{eqnarray*}
D\Psi(0,0)\cdot v&=&\dfrac{d\Psi}{ds}(0,sv)\arrowvert_{s=0}\\
&=&\dfrac{d}{ds}\Big(H_{\Sigma_{sv}}-\dfrac{1}{|\Sigma|}\int_{\Sigma}H_{\Sigma_{sv}\,d\sigma}\Big)\big\arrowvert_{s=0}\\
&=&-Lv-\dfrac{1}{|\Sigma|}\int_{\Sigma}Lv\,d\sigma\\
&=&-\Delta_{\Sigma}v,
\end{eqnarray*}
where the last equality follows from the fact that $L=\Delta_{\Sigma}.$

Since $\Delta_{\Sigma}:X\longrightarrow Y$ is a linear isomorphism, we have, by the implicit function theorem, that there exist $0<\epsilon_1<\epsilon$ and $u(t)=u(t,\cdot)\in B(0,\delta)$ for $t\in (-\epsilon_1,\epsilon_1)$ such that 
$$
u(0)=0\ \ \textrm{and}\ \ \Psi(t,u(t))=0,\forall t\in (-\epsilon_1,\epsilon_1).
$$

Thus, defining $w(t,x)=u(t,x)+t$, for $(t,x)\in(-\epsilon_1,\epsilon_1)\times\Sigma$,  we have that all the surfaces $\Sigma_t=\{\textrm{exp}_x(w(t,x)\nu(x)):x\in\Sigma\}$ have constant mean curvature. It is easy to see that $w(t,x)$ satisfies all the conditions stated in the proposition.
\end{proof} 

Let $\nu(t)$ denote the unit normal vector field to $\Sigma_t$ such that $\nu(0)=\nu$. In our convention, the mean curvature $H(t)$ of $\Sigma_t$ satisfies $\vec{H}(t)=-H(t)\nu(t)$, where $\vec{H}(t)$ is the mean curvature vector of $\Sigma_t$. In this case, we have 
\begin{equation}
\dfrac{d}{dt}|\Sigma_t|_g=H(t)\int_{\Sigma}\big\langle\nu(t),\dfrac{\partial f}{\partial t}(t,\cdot)\big\rangle \,d\sigma, \label{equation 1}
\end{equation}
where $f(t,x)=\textrm{exp}_x(w(t,x)\,\nu(t))$, $x\in \Sigma.$ Notice that $\dfrac{\partial f}{\partial t}(0,x)=\nu(x)$, so we can suppose, decreasing $\epsilon$ if necessary, that $\int_{\Sigma}\big\langle\nu(t),\dfrac{\partial f}{\partial t}(t,\cdot)\big\rangle\,\,d\sigma$ is positive for all $t\in (-\epsilon,\epsilon)$. Moreover, we can assume that $|\Sigma|_g\leqslant|\Sigma_t|_g$ for all $t\in(-\epsilon,\epsilon)$, because $\Sigma$ is locally area-minimizing.

Before we prove the next proposition, we will recall some facts about the Yamabe problem on manifolds with boundary which was first studied by J. F. Escobar \cite{E}. Let $(M^n,g)$ be a compact Riemannian manifold with boundary $\partial M\neq\emptyset$. It is a basic fact that the existence of a metric $\overline{g}$ in the conformal class of $g$ having scalar curvature equal to $C\in\mathbb{R}$ and the boundary being a minimal hypersurface is equivalent to the existence of a positive smooth function $u\in C^{\infty}(M)$ satisfying

\begin{equation}
\left\{
\begin{aligned} 
\Delta_gu-\dfrac{n-2}{4(n-1)}R_gu+\dfrac{n-2}{4(n-1)}Cu^{(n+2)/(n-2)}&=0\ \ \textrm{on}\ \ M\\
\dfrac{\partial u}{\partial\eta}+\dfrac{n-2}{2(n-1)}H_gu&=0\ \ \textrm{on}\ \ \partial M\label{Yamabe}
\end{aligned}
\right.
\end{equation}
where $\eta$ is the outward normal vector with respect to the metric g.

If $u$ is a solution of the equation above, then $u$ is a critical point of the following functional
$$
Q_g(\phi)=\dfrac{\int_M(|\nabla_g\phi|_g^2+\frac{n-2}{4(n-1)}\,R_g\,\phi^2)\,dv+\frac{n-2}{2(n-1)}\int_{\partial M}H_g\,\phi^2\,d\sigma}{(\int_M|\phi|^{2n/(n-2)}\,dv)^{(n-2)/n}}.
$$

The Sobolev quotient $Q(M)$ is then defined by
$$
Q(M)=\inf\{Q_g(\phi):\phi\in C^1(M), \phi\neq 0\}.
$$

It is a well known fact that $Q(M)\leqslant Q(S_{+}^n)$, where $S_{+}^n$ is the upper standard hemisphere, and that if $Q(M)<Q(S_+^n)$, then there exists a smooth minimizer for the functional above. This function turns out to be a positive solution of (\ref{Yamabe}), with a constant $C$ that has the same sign as $Q(M)$.

\begin{Prop}\label{prop3}
There exists $0<\epsilon_1<\epsilon$ such that $H(t)\leqslant 0$  $\forall t\in [0,\epsilon_1)$. 
\end{Prop}
\begin{proof}
Suppose, by contradiction, that there exists a sequence $\epsilon_k\rightarrow 0$, $\epsilon_k>0,$ such that $H(\epsilon_k)>0$ for all $k$. Consider $(V_k,g_k)$, where $V_k=[0,\epsilon_k]\times\Sigma$ and $g_k$ is the pullback of the metric $g$ by $f\arrowvert_{V_k}:V_k\longrightarrow M.$ Therefore, $V_k$ is a compact three-manifold with boundary satisfying
\begin{itemize}
\item $R_{g_k}\geqslant -2$.
\item The mean curvature of $\partial V_k$ with respect to inward normal vector, denoted by $H_{g_k}$, is nonnegative. More precisely, $\partial V_k=\Sigma\cup\Sigma_{\epsilon_k}$ where $\Sigma$ is a minimal surface and $\Sigma_{\epsilon_k}$ has positive constant mean curvature.
\item $|\Sigma|_{g_k}=4\pi(g(\Sigma)-1).$ 
\end{itemize}

\begin{Claim}
For k sufficiently large, we have $Q(V_k)<0$. In particular, this implies $Q(V_k)<Q(S_+^3).$
\end{Claim}

\begin{proof}
By Proposition \ref{Prop1}, we have $R_g=-2$ on $\Sigma$. Therefore, by continuity, we have $-2\leqslant R_{g_k}\leqslant -1$ on $V_k$ for $k$ sufficiently large. Choosing $\phi=1$, we obtain
\begin{eqnarray*}
Q_{g_k}(\phi)&=&\dfrac{\frac{1}{8}\int_{V_k}R_{g_k}\,dv_k+\frac{1}{4}\int_{\partial V_k}H_{g_k}\,d\sigma_k}{Vol(V_k)^{1/3}}\\
&\leqslant&\dfrac{-\frac{1}{8}Vol(V_k)+\frac{1}{4}H(\epsilon_k)|\Sigma_{\epsilon_k}|_{g_k}}{Vol(V_k)^{1/3}}.
\end{eqnarray*}

Since $\dfrac{\partial f}{\partial t}(0,x)=\nu(x)$ and the stability operator of $\Sigma$ is equal to $\Delta_{\Sigma}$, we obtain that $\dfrac{d}{dt}H(t)\arrowvert_{t=0}=0.$ Therefore, we conclude that $H(\epsilon_k)=O(\epsilon_k^2)$ because $H(0)=H_{\Sigma}=0$. Moreover, if $V(t)=[0,t]\times\Sigma$ and $g_t=(f\arrowvert_{V(t)})^*g$, we have that
\begin{eqnarray*}
\Vol(V(t))&=&\Vol(V(t),g_t)\\
&=&\int_{[0,t]\times\Sigma}(f\arrowvert_{V(t)})^*dv\\
&=&\int_{[0,t]\times\Sigma}h(s,x)\,ds\wedge d\sigma\\
&=&\int_0^t\int_{\Sigma}h(s,x)\,d\sigma\,ds,
\end{eqnarray*}
where $h$ is defined  by $h(s,x)=dv(\frac{\partial f}{\partial s}(s,x),Df(s,x)\,e_1, Df(s,x)\,e_2)$ and \linebreak$\{e_1,e_2\}\subset T\Sigma$ is a positive orthonormal basis with respect to the induced metric on $\Sigma$.
From this, we get
$$
\dfrac{d}{dt}\Vol(V(t))\arrowvert_{t=0}=\int_{\Sigma}h(0,x)\,d\sigma.
$$

Since $\dfrac{\partial f}{\partial s}(0,x)=\nu(x)$, we have $h(0,x)=1$. Hence, $\dfrac{d}{dt}\Vol(V(t))\arrowvert_{t=0}=|\Sigma|_g$. From this, we obtain that $\Vol(V_k)=\epsilon_k|\Sigma|_{g_k}+O(\epsilon_k^2).$ Finally, it is easy to see that for $k$ sufficiently large we have $Q(V_k)\leqslant Q_{g_k}(\phi)<0.$ This concludes the proof of the claim. 
\end{proof}

Choose $k$ sufficiently large such that $Q(V_k)<0$. Thus, we have that there exists a positive function $u\in C^{\infty}(V_k)$ such that the metric $\overline{g}=u^4g_k$ satisfies
$$
R_{\overline{g}}=C<0,\ C\in\mathbb{R},\ \ \textrm{on}\ \ V_k \ \ \textrm{and}\ \ H_{\overline{g}}=0\ \ \textrm{on}\ \ \partial V_k.
$$
After scaling the metric $\overline{g}$ if necessary, we can suppose that $C=-2$.

In analytic terms, this means that $u$ solves
\begin{equation}
\left\{
\begin{aligned}
\Delta_{g_k}u-\dfrac{1}{8}R_{g_k}u-\dfrac{1}{4}u^5&=0\ \ \textrm{on}\ \ V_k\\
\dfrac{\partial u}{\partial \eta}+\dfrac{1}{4}H_{g_k}u&=0\ \ \textrm{on}\ \ \partial V_k.\label{Yamabe2}
\end{aligned}
\right.
\end{equation}

By (\ref{Yamabe2}) and the fact that $R_{g_k}\geqslant -2$, we have 
$$
\Delta_{g_k}u+\dfrac{1}{4}u-\dfrac{1}{4}u^5\geqslant 0\ \ \textrm{on} \ \ V_k.
$$

Consider $x_0\in V_k$ such that $u(x_0)=\displaystyle\max_{x\in V_k}u(x).$ If $x_0\in V_k\setminus\partial V_k$, then we get
$$
\dfrac{1}{4}u(x_0)\geqslant\dfrac{1}{4}u(x_0)^5.
$$
Thus, $u(x_0)\leqslant 1$. It follows, by the maximum principle, that either $u\equiv 1$ or $u<1$. The first possibility does not occur because the mean curvature of $\Sigma_{g_k}$ with respect to $g_k$ is positive and with respect to $\overline{g}$ is equal to zero. It follows that $u<1$.

Now, suppose $x_0\in\partial V_k$. If $u(x_0)\geqslant 1$, we obtain, by the Hopf's boundary point lemma, that either $u$ is constant or $\dfrac{\partial u}{\partial\eta}(x_0)>0$. The first possibility does not occur by the same argument used in the interior maximum case. But, since $H_{g_k}\geqslant 0$, (\ref{Yamabe2}) implies that $\dfrac{\partial u}{\partial \eta}(x_0)\leqslant 0$. Thus, the second possibility is also not possible. Hence, $u(x_0)<1.$

Therefore, we conclude that $u(x)<1$ for all $x\in V_k$. From this, we obtain that $|\Sigma|_{\overline{g}}<|\Sigma|_{g_k}=4\pi(g(\Sigma)-1)$.



Finally, denote by $\mathfrak{I}(\Sigma)$ the isotopy class of $\Sigma$ in $V_k$. Observe that $\Sigma$ is incompressible in $V_k$. Moreover, we have that $V_k$ is orientable and irreducible and does not contain any one-sided compact embedded surface. Since $H_{\overline{g}}=0$, we can directly apply the version for three-manifolds with boundary  of the Theorem 5.1 in \cite{HS}, (see also \cite{MSY}), to obtain a compact embedded surface $\overline{\Sigma}\in\mathfrak{I}(\Sigma)$ such that
$$
|\overline{\Sigma}|_{\overline{g}}=\displaystyle\inf_{\hat{\Sigma}\in\mathfrak{I}(\Sigma)}|\hat{\Sigma}|_{\overline{g}}.
$$

Therefore, $|\overline{\Sigma}|_{\overline{g}}\leqslant|\Sigma|_{\overline{g}}<4\pi(g(\Sigma)-1)$. But this is a contradiction with (\ref{inequality}), since we have proven, by using the lower bound $R_{\overline{g}}\geqslant-2$ and the second variation of  area, that we must have $|\overline{\Sigma}|_{\overline{g}}\geqslant 4\pi(g(\Sigma)-1)$. This concludes the proof of the proposition.
\end{proof}
Next, we will conclude the rigidity in the Theorem \ref{Thm1}. Observe that the Proposition \ref{prop3} implies $\dfrac{d}{dt}|\Sigma_t|_g\leqslant 0$ for all $t\in[0,\epsilon_1)$. Thus, $|\Sigma_t|_g\leqslant|\Sigma|_g$ for all $t\in[0,\epsilon_1)$ and this implies $|\Sigma_t|_g=|\Sigma|_g$ for all $t\in[0,\epsilon_1)$ because $\Sigma$ is locally area-minimizing. Therefore, by Proposition \ref{Prop1}, we have that $\Sigma_t$ is totally geodesic and $\Ric(\nu(t),\nu(t))=0$ on $\Sigma_t$ for all $t\in[0,\epsilon_1)$. In particular, we have that all the surfaces $\Sigma_t$ are minimal and the stability operator of $\Sigma_t$, denoted by $L_{\Sigma_t}$, is equal to $\Delta_{\Sigma_t}.$

Define $\rho(t)=\rho(t,x)=\big\langle\nu(t,x),\dfrac{\partial f}{\partial t}(t,x)\big\rangle$. We have
$$
L_{\Sigma_t}\rho(t)=-H^{\prime}(t),
$$
so $\Delta_{\Sigma_t}\rho(t)=0.$ Thus, $\rho(t)$ does not depend on $x$.

Since $\Sigma_t$ is totally geodesic, we have that $\nabla_{\frac{\partial f}{\partial x_i}}\nu(t)=0$ for all $i=1,2$, where $(x_1,x_2)$ are local coordinates on $\Sigma$. Moreover, by the fact that $\langle\nu(t),\nu(t)\rangle=1$ we have that $\nabla_{\frac{\partial f}{\partial t}}\nu(t)$ is tangent to $\Sigma_t$. Hence, it follows that
\begin{eqnarray*}
\big\langle\nabla_{\frac{\partial f}{\partial t}}\nu(t),\frac{\partial f}{\partial x_i}\big\rangle&=&\frac{\partial}{\partial t}\big\langle\nu(t),\frac{\partial f}{\partial x_i}\big\rangle-\big\langle\nu(t),\nabla_{\frac{\partial f}{\partial t}}(\partial f/\partial x_i)\big\rangle\\
&=&-\big\langle\nu(t),\nabla_{\frac{\partial f}{\partial x_i}}(\partial f/\partial t)\big\rangle\\
&=&-\frac{\partial}{\partial x_i}\rho(t)\\
&=&0,
\end{eqnarray*}
for all $i=1,2$. Hence, $\nabla_{\frac{\partial f}{\partial t}}\nu(t)=0$. This means that, for all $x\in\Sigma$, $\nu(t,x)$ is a parallel vector field along the curve $\alpha_x:[0,\epsilon_1)\longrightarrow M$ given by $\alpha_x(t)=f(t,x)=\textrm{exp}_x(w(t,x)\nu(x))$. 

Observe that $D(\exp_x)_{w(t,x)\nu(x)}(\nu(x))$ is also a parallel vector field along the curve $\alpha_x$. Thus, $\nu(t,x)=D(\exp_x)_{w(t,x)\nu(x)}(\nu(x))$ because $w(0,x)=1$ by Proposition \ref{Prop1}. From this, we conclude that $\rho(t)=\dfrac{\partial w}{\partial t}(t,x)$.

By Proposition \ref{Prop1}, we have 
$$
\int_\Sigma (w(t,x)-t)\,d\sigma=0,
$$
so
$$
\int_{\Sigma}\dfrac{\partial w}{\partial t}(t,x)\,d\sigma=|\Sigma|_g.
$$

Therefore, since $\dfrac{\partial w}{\partial t}(t,x)$ does not depend on $x$, we get $\dfrac{\partial w}{\partial t}(t,x)=1$. This implies that $w(t,x)=t$ for all $(t,x)\in[0,\epsilon_1)\times\Sigma$ because $w(0,x)=0$. Thus, we conclude that $f(t,x)=\exp_x(t\nu(x))$ and, since $\Sigma_t$ are totally geodesic, the pullback of $g$ by $f|_{[0,\epsilon_1)\times\Sigma}$ is the product metric $dt^2+g_{\Sigma}$, where $g_{\Sigma}$ is the induced metric on $\Sigma$.

Arguing similarly for $t\leqslant 0$, we obtain the following proposition which is the rigidity in Theorem \ref{Thm1}.

\begin{Prop}\label{Prop4}
If $\Sigma$ attains the equality in (\ref{inequality}), then $\Sigma$ has a neighborhood which is isometric to $((-\epsilon,\epsilon)\times\Sigma,dt^2+g_{\Sigma})$, where $\epsilon>0$ and $g_{\Sigma}$ is the induced metric on $\Sigma$ which has constant Gauss curvature equal to -1.
\end{Prop}

Now, we will prove the Corollary \ref{Cor1}. Suppose $\Sigma$ minimizes area in its homotopy class and $\Sigma$ attains the equality in (\ref{Thm1}). Define $f:\mathbb{R}\times\Sigma\longrightarrow M$ by $f(t,x)=\exp_x(t\nu(x))$, where $\nu$ is the unit normal vector field to $\Sigma$.

\begin{Prop}\label{Prop5} 
$f:(\mathbb{R}\times\Sigma,dt^2+g_{\Sigma})\longrightarrow (M,g)$ is a local isometry.
\end{Prop}
\begin{proof}
Consider $A=\{t>0:f\arrowvert_{[0,t]\times\Sigma}\ \, \textrm{is a local isometry}\}$. By Proposition \ref{Prop4}, this set is nonempty. Moreover, $A$ is closed. Let us prove that $A$ is open. Given $t\in A$, consider the immersed surface $\Sigma_t=\{\textrm{exp}_x(t\nu(x)):x\in\Sigma\}$ with the metric induced by $f$. We have that $\Sigma_t$ is homotopic to $\Sigma$ and $|\Sigma_t|=|\Sigma|$. Hence, $\Sigma_t$ minimizes area in its homotopy class and attains the equality in (\ref{inequality}). Therefore, by Proposition \ref{Prop4}, we conclude that there exists $\epsilon>0$ such that $f\arrowvert_{[0,t+\epsilon]\times\Sigma}$ is a local isometry. It follows that $A$ is open and consequently $f|_{[0,\infty)\times\Sigma}$ is a local isometry. Arguing similarly for $t<0$, we conclude the proposition.
\end{proof}

To conclude the Corollary \ref{Cor1} (cf. pg 2), observe that the Proposition above implies that $f:(\mathbb{R}\times\Sigma,dt^2+g_{\Sigma})\longrightarrow (M,g)$ is a covering map.

\section{Proof of Theorem \ref{Cor2}}

Let $\partial\Omega^{(1)}$ be a connected component of $\partial\Omega$ which is isometric to $(\Sigma,g_{\Sigma})$. Consider $\alpha=\inf\{|\hat{\Sigma}|_g:\hat{\Sigma}\in\mathfrak{I}(\partial\Omega^{(1)})\}$, where $\mathfrak{I}(\partial\Omega^{(1)})$ is the isotopy class of $\mathfrak{I}(\partial\Omega^{(1)})$. By hypothesis, $\partial\Omega^{(1)}$ is incompressible in $\Omega$, $H_g\geqslant 0$ and  $\Omega$ is orientable and irreducible and does not contain one-sided compact embedded surfaces. Therefore, we can apply the version for three-manifolds with boundary of the Theorem 5.1 in \cite{HS} (see also \cite{MSY}), to obtain a compact embedded surface $\overline{\Sigma}\in\mathfrak{I}(\partial\Omega^{(1)})$ such that $|\overline{\Sigma}|=\alpha.$ Note that $\overline{\Sigma}\in\mathfrak{I}(\partial\Omega^{(1)})$ implies $\overline{\Sigma}$ has genus equal to $g(\Sigma).$

Since all connected components of $\partial\Omega$ have nonnegative mean curvature, it follows from the maximum principle that either $\overline{\Sigma}$ is a boundary component of $\Omega$ or $\overline{\Sigma}$ is in the interior of $\Omega$. If $\overline{\Sigma}$ is in the interior of $\Omega$, then we obtain, by Theorem \ref{Thm1}, that $|\overline{\Sigma}|\geqslant 4\pi(g(\Sigma)-1)$ since $R_g\geqslant -2$ and $\overline{\Sigma}$ has genus equal to $g(\Sigma)$. On the other hand, we have $|\partial\Omega^{(1)}|=4\pi(g(\Sigma)-1)$ because $\partial\Omega^{(1)}$ is isometric to $(\Sigma,g_{\Sigma})$. From this, we get $|\overline{\Sigma}|=4\pi(g(\Sigma)-1).$ Now, if $\overline{\Sigma}$ is a boundary component of $\Omega$, then we have that $\overline{\Sigma}$ is a minimal surface because  $\overline{\Sigma}$ is area-minimizing and, by hypothesis, has nonnegative mean curvature. This implies, using Theorem \ref{Thm1}, that $|\overline{\Sigma}|\geqslant 4\pi(g(\Sigma)-1)$. Again we conclude that $|\overline{\Sigma}|=4\pi(g(\Sigma)-1)$. It follows from the previous arguments that we can suppose $\overline{\Sigma}=\partial\Omega^{(1)}$.

By the proof of the rigidity in Theorem \ref{Thm1}, we have that there exists $\epsilon>0$ such that the normal exponential map $f:[0,\epsilon)\times\overline{\Sigma}\longrightarrow\Omega$ defined by $f(t,x)=\exp_x(t\nu(x))$, where $\nu$ is the inward normal vector, is an injective local isometry.

Define 
$
l=\sup\{t>0:f(t,x)=\exp_x(t\nu(x))\ \, \textrm{is defined on}\ \,[0,t)\times\overline{\Sigma}\linebreak
\textrm{and is an injective local isometry}\}.
$
Since $\Omega$ is complete, we have that the normal geodesics to $\overline{\Sigma}$ extend to $t=l$. Thus, $f$ is defined on $[0,l]\times\overline{\Sigma}$. By continuity and the definition of $l$, we obtain that $f:[0,l]\times\overline{\Sigma}\longrightarrow\Omega$ is a local isometry. In particular, by continuity, the immersion $f:\overline{\Sigma}_l\longrightarrow\Omega$ is totally geodesic, where $\overline{\Sigma}_l=\{l\}\times\overline{\Sigma}.$

Using again the maximum principle, we obtain that either $f(\overline{\Sigma}_l)$ is a boundary component of $\Omega$, different from $\overline{\Sigma}$ because of the injectivity of $f$ on $[0,l)\times\overline{\Sigma}$, or $f(\overline{\Sigma}_l)$ is in the interior of $\Omega$.

Suppose $f(\overline{\Sigma}_l)$ is a boundary component of $\Omega$. Since $f$ is a local isometry on $[0,l]\times\overline{\Sigma}$, we have $\dfrac{\partial f}{\partial t}(l,x)$ is a unit normal vector to $\overline{\Sigma}_l$. It follows from this that $f:\overline{\Sigma}_l\longrightarrow\Omega$ is injective because $\overline{\Sigma}_l$ is a boundary component of $\Omega$. Thus, $f:[0,l]\times\overline{\Sigma}\longrightarrow\Omega$ is an injective local isometry. Since $\Omega$ is connected, we obtain $f([0,l]\times\overline{\Sigma})=\Omega$. Therefore, we have that $\Omega$ is isometric to $[0,l]\times\overline{\Sigma}.$

Let us analyze the case where $f(\overline{\Sigma}_l)$ is in the interior of $\Omega$. First, we have that $f:\overline{\Sigma}_l\longrightarrow\Omega$ cannot be injective. In fact, suppose $f:\overline{\Sigma}_l\longrightarrow\Omega$ is injective. Thus, by the rigidity in the Theorem \ref{Thm1}, there exists $\epsilon>0$ such that $f:[0,l+\epsilon)\longrightarrow\Omega$ is an injective  local isometry which is a contradiction because of the maximality of $l$. Therefore, there exist $x,y\in\overline{\Sigma}, x\neq y$, such that $f(l,x)=f(l,y)$. We have $Df(l,x)(T\overline{\Sigma}_l)=Df(l,y)(T\overline{\Sigma}_l)$, since otherwise $f$ would not be injective on $[0,l)\times\overline{\Sigma}$. This implies $\dfrac{\partial f}{\partial t}(l,x)=-\dfrac{\partial f}{\partial t}(l,y)$. Thus, since $f:\overline{\Sigma}_l\longrightarrow\Omega$ is totally geodesic, there exist neighborhoods of $x$ and $y$ in $\overline{\Sigma}_l$, respectively, such that the images by $f$ of these neighborhoods coincide. We conclude that $\hat{\Sigma}_l=f(\overline{\Sigma}_l)$ is a one-sided embedded compact surface in $\Omega$. But, this is a contradiction because, by hypothesis, $\Omega$ does not contain any one-sided embedded compact surface. This concludes the proof of Theorem \ref{Cor2}.
\begin{flushleft}
\textsc{Ivaldo Nunes}

\textsc{Instituto Nacional de Matem\' atica Pura e Aplicada (IMPA)}

\textsc{Estrada Dona Castorina 110, 22460-320, Rio de Janeiro-RJ, Brazil}

\textit{Email adress:} ivaldo82@impa.br
\end{flushleft}

\small {}
\end{document}